\documentclass[runningheads]{llncs}

\usepackage{amsmath}
\usepackage{amssymb}

\usepackage{graphicx}

\usepackage{enumerate}
\usepackage[hyphens]{url}

\usepackage{hyperref}
\hypersetup{colorlinks=true}
\usepackage[nameinlink]{cleveref}

\begin{document}
\title{Kernel Density Estimation Bias under Minimal Assumptions}
%
%
\author{Maciej Skorski \inst{1}}
\authorrunning{M. Skorski}
%
\institute{DELL\\
\email{maciej.skorski@gmail.com}}
\maketitle              
\begin{abstract}
Kernel Density Estimation is a very popular technique of approximating a density function from samples. The accuracy is generally well-understood and depends, roughly speaking, on the kernel decay and local smoothness of the true density. However concrete statements in the literature are often invoked in very specific settings (simplified or overly conservative assumptions) or miss important but subtle points (e.g. it is common to heuristically apply Taylor's expansion globally without referring to compactness). 

The contribution of this paper is twofold
\begin{enumerate}[(a)]
\item we demonstrate that it is necessary to keep a certain balance between the
\emph{kernel decay} and \emph{magnitudes of bandwidth eigenvalues}; otherwise, regardless of kernel smoothness and moments (!), the estimates are not bounded.
\item we give a rigorous derivation of bounds with explicit constants for the bias, under possibly minimal assumptions. This connects the kernel decay, bandwidth norm, bandwidth determinant and (local) density smoothness.
\end{enumerate}
It has been folklore that the issue with Taylor's formula can be fixed with more complicated assumptions on the density (for example p. 95 of "Kernel Smoothing" by Wand and Jones); we show that 
this is actually not necessary and can be handled by the kernel decay alone.
\end{abstract}
\section{Introduction}

\subsection{Kernel Density Estimation}
\paragraph{Density estimation by convolutions}
Density estimation is the fundamental problem of approximating a probability density function
$f:\mathbb{R}^d\rightarrow \mathbb{R}$ given a set of $n$ iid samples $\mathcal{D}\subset \mathbb{R}^d$. The popular approach, called \emph{Kernel Density Estimation}, uses
a convolution of a suitable filter $K_h$ (called \emph{kernel}) with the sample distribution  $f_{\mathcal{D}} = |\mathcal{D}|^{-1}\sum_{\mathbf{x}\in\mathcal{D}}\delta_{\mathbf{x}}$. Formally, the KDE estimator is defined by
\begin{align}\label{eq:kde}
\hat{f}(\mathbf{x}') = K\star f_{\mathcal{D}}(\mathbf{x}')  = \int K_h(\mathbf{x}'-\mathbf{x}) f_{\mathcal{D}}(\mathbf{x})\,\mbox{d} \mathbf{x}
\end{align}
and in this form is credited to Rosenblatt and Parzen~\cite{RosenblattM1956,parzen1962}. Usually one uses rescaled versions of a base kernel 
\begin{align}\label{eq:scale}
K_h(\mathbf{u}) = |h|^{-1} \cdot K(h^{-1}\mathbf{u})
\end{align}
where the scale parameter $h$ is a $d\times d$ invertible matrix  called \emph{bandwidth} and $|h|$ is the matrix determinant (for simplicity one often considers diagonal $h$).
Under certain assumptions on the kernel (rapid decay, moments) and the density (smoothness),
the KDE estimator is consistent asymptotically, that is when $h\to 0$. Intuitively, the convergence follows as for $\mathbf{x}'$ close to $\mathbf{x}$ we have $\mathbf{f}(\mathbf{x})\approx \mathbf{f}(\mathbf{x'})$ by the smoothness of $f$, and for larger $\mathbf{x}$ the possible bias is penalized by the scaled kernel as $|h^{-1}(\mathbf{x}'-\mathbf{x})|$ is big for small $h$. Specific bounds depends on the kernel and local smoothnes of $f$.

\paragraph{Estimator Accuracy}
The variance of the estimator is quite easy to compute
\begin{align}
\mathsf{Var}(\hat{f}(\mathbf{x'})) = |\mathcal{D}|^{-1}\cdot\mathsf{Var}_{\mathbf{x}\sim f}\left[K_h(\mathbf{x}'-\mathbf{x})\right]
\end{align}
and (under some assumptions on $f$) is of order $O(|\mathcal{D}|^{-1}\|h\|^{-d})$ with the hidden constant dependent on $K$. In turn, bias is obtained by exchanging expectation and the convolution integral
\begin{align}
\mathsf{bias}(\hat{f}(\mathbf{x}')) = \mathbb{E}_{\mathcal{D}}\left[K\star f_{\mathcal{D}}(\mathbf{x}')-f(\mathbf{x}')\right] = K_h\star f(\mathbf{x}')-f(\mathbf{x}') 
\end{align}
Intuitively, it captures by how much the convolution perturbs the density; this in turn depends on how the kernel interacts with the local series expansion of $f$. Expanding $f$ around $\mathbf{x}'$ and parametrizing $\mathbf{x}=\mathbf{x}'-h\mathbf{u}$
one obtains a series $\mathsf{bias}(\hat{f}(\mathbf{x}'))=\sum_{k}\mu_k(\mathbf{x}',h)$ where
\begin{align}\label{eq:moments}
\mu_k(\mathbf{x}',h) = \left\{\begin{array}{rl}
 f(\mathbf{x}')\cdot (\int K(\mathbf{u})-1)\mbox{d}\mathbf{u} & \quad k=0 \\
     \frac{(-1)^k}{k!} \int K(\mathbf{u}) D^{k} f(\mathbf{x}')((h\mathbf{u})^{(k)})\,\mbox{d}\mathbf{u} & \quad k=1,2,\ldots
\end{array}\right.
\end{align}
the $k$-th derivative $D^k$ is understood as a $k$-linear map
from $(\mathbb{R}^d)^k$ to $\mathbb{R}^d$ and $\mathbf{v}^{(k)}$ denotes the vector $\mathbf{v}$ stacked $k$-times; here one needs to some assumptions on
the kernel $K$ and derivatives $D^j$ to guarantee that the integrals exist. 

In general $\mu_k(\mathbf{x}',h) = O(\mu_K(k) \|h\|^k) = O(\|h\|^k)$, so one designs the filter to eliminate low-order terms: 
\begin{enumerate}[(a)]
\item (unit mass) when $\int K(\mathbf{u}) = 1$, the bias is of order $O(\|h\|)$
\item (symmetry) if in addition $K(\mathbf{u})=K(\mathbf{-u})$, the bias improves to $O(\|h\|^2)$\footnote{$\mu_1$ is a weighted sum of terms $\int\ K(\mathbf{u})\mathbf{u}_i\,\mbox{d}\mathbf{u}$, which are zero when $K$ is symmetric.}.
\end{enumerate}
The best, over the choice of $h$, MSE error equals then (pointwise, for fixed $\mathbf{x}'$)
 \begin{align}
     \mathsf{MSE} = O\left(n^{-\frac{4}{4+d}}\right),\quad \|h\| = \Theta\left(n^{-\frac{1}{4+d}}\right)
 \end{align}
This improves upon histograms (they have error $O\left(n^{-\frac{2}{2+d}}\right)$). Cacoullos~\cite{multi_KDE}
gives a rigorous derivation of the bias and variance formulas for diagonal $h$.

\paragraph{Better accuracy with higher-order kernels}
One can farther reduce bias by eliminating more of the expansion terms. Such kernels are also called \emph{higher-order kernels}
and compensate the negative impact of dimension $d$ on the variance (curse of dimensionality).
If the property $\mu_k=0$ holds for $k=1,\ldots,v-1$ one says the kernel is of order $v$; the bias is of order $\|h\|^{v}$ which (for the optimized bandwith) gives the mse error of order $O\left(n^{-\frac{2v}{2v+d}}\right)$~\cite{KDE_lectures2}.
Higher-order kernels can be built as products of single-dimension higher-order kernels; the problem of developing one-dimensional filters from Taylor expansions was studied in~\cite{Moller1997}.

\subsection{Contribution of this paper}

The
fundamental properties of kernel estimators, including bias and variance, are generally well understood. However the concrete statements in the existing literature are based on various assumptions; sometimes they are overly simplistic, sometimes too conservative, and finally sometimes important assumptions are ignored. We mention few prominent examples, to be specific: 
\begin{itemize}
\item Bandwidth is scalar or diagonal~\cite{multi_KDE}, or 
is given by rescaling a fixed matrix~\cite{pmlr-v70-jiang17b,multi_KDE}
\item For second-order kernels, smoothness of the density
of order $3$ or higher is assumed~\cite{sample_kde1,multi_KDE}
\item Taylor's expansion is used \emph{globaly} which suggest that the kernel decay is not needed~\cite{KDE_lectures,KDE_lectures2} without referring to compactness or taking compact arguments to hold globally~\cite{DUONG2005417}. In fact without sufficient decay the estimates are not even bounded (we will discuss a general example).
\end{itemize}
The purpose of this paper is to give a rigorous bounds on the bias, under minimal constraints on the bandwidth matrix and the kernel decay. Particularly, we discuss what happens when the bandwidth elements goes to zero at different rates.

\section{Results}

\subsection{Necessary kernel decay and bandwidth eigenvalues balance}

The $j$-th moment of the kernel $K$ is defined as
$\mu_K(k) = \int |\mathbf{u}|^j|K(\mathbf{u})|\mbox{d}\mathbf{u}$. The following construction shows
that, to reconstruct the density from its behavior in a fixed neighborhood, the kernel decay and discrepancy of eigenvalues of $h$ must be balanced. This is true \emph{regardless} of smoothness and moments of $K$ (note that bounded moments do not imply decay!).

\begin{theorem}[Lower bound on bias in terms of kernel decay and bandwidth eigenvalues]
For any $\ell \geqslant 0, p\geqslant 0$ there exists a radial kernel $K$ on $\mathbb{R}^d$ which is infinitely differentiable, has finite $\ell$ first moments, and decay rate at infinity not faster than $\|\mathbf{u}\|^{-p}$ with the following property:
over the class of densities $f$ with given behavior on the unit ball
\begin{align*}
\mathcal{F} = \left\{f: f(\mathbf{u}) = f_0(\mathbf{u})\quad \forall \mathbf{u}:|\mathbf{u}|\leqslant 1  \right\},\quad \textrm{ for any } f_0:\ \int_{|\mathbf{u}|\leqslant 1} f_0(\mathbf{u})\mathrm{d}\mathbf{u}<1
\end{align*}
the density estimation at $\mathbf{0}$ is lower-bounded by 
 \begin{align}\label{eq:lower_bound}
 \max_{f\in \mathcal{F}}  \left[ K_h\star f(\mathbf{0}) \right] = \Omega(1)\cdot |\lambda_{1}|^{p}/\prod_{i=1}^d|\lambda_i|
\end{align}
where $\{\lambda_{i}\}_i$ are eigenvalues of $h$ ordered so that $|\lambda_1|\geqslant \ldots \geqslant |\lambda_d|$.
\end{theorem}
\begin{proof}
Consider a non-negative "radial" kernel $K(\mathbf{u}) = k(|\mathbf{u}|)$
on $\mathbb{R}^d$ where $k$ is a non-negative real function such that
\begin{align}
\sup_{r\in\mathbb{Z}} r^p K(r) = \Omega(1)
\end{align}
for some fixed $p\geqslant 0$, the supremum being over integers. For example, let $\Psi(r) = \exp(-1/(1-r)^2)\mathbf{1}_{-1\leqslant r\leqslant 1}(r)$ be the standard bump function.
Now for some constant $c$ consider $$k(r) = c\cdot \sum_{n\in\mathbb{Z},|n|\geqslant 2}  |n|^{-p} \cdot \Psi (2 |n|^{p+\ell+d+1} r-n)$$
the sum of shifted and rescaled bump functions - the $n$-th is component centered at $n$ with the interval width $|n|^{-p-d-\ell-1}\leqslant 1$ and the spike of magnitude $|n|^{p}$. Clearly $k$ is analytic because each point is covered by finitely many smooth components (actually by at most one)
Moreover, $k$ has integrable moments up to order $d+\ell-1$
\begin{align*} 
\int |r^{j} k(r)| \mbox{d}  r \leqslant \sum_{n\in\mathbf{Z},|n|\geqslant 2} n^{p+j}\cdot n^{-p-d-\ell-1} < \infty,\quad j=0,\ldots,d+\ell-1.
\end{align*}
It is well-known that for radial functions $K(\mathbf{u}) = k(|\mathbf{u}|)$ it holds $\int |\mathbf{u}|^j |K(\mathbf{u})|\mbox{d}\mathbf{u} = O(1) \int_{0}^{\infty} r^{d+j-1} k(r)\mbox{d}r$ (by the spherical parametrization). Therefore $K$ defined from $k$ our $k$ is indeed integrable and has all moments up to $\ell$; by manipulating $c$ we can normalize the integral to $1$. Note also that $K(\mathbf{u}) = k(|\mathbf{u}|)$ is infinitely differentiable in $\mathbf{u}$,
also at $0$ because $K=0$ in the neighborhood of zero by definition. Now, since $K$ and $f$ are positive 
\begin{align*}
K_h\star f(\mathbf{0}) = |h|^{-1}\int K(-h^{-1}\mathbf{u}) f(\mathbf{u})\mathrm{d}\mathbf{u} = \Omega(1)\cdot |h|^{-1}
\int_{|\mathbf{u}|>1} |h^{-1}\mathbf{u}|^{-p} f(\mathbf{u})\mathrm{d}\mathbf{u} 
\end{align*}
The class $\mathcal{F}$ represents all functions with \emph{same behavior on the unit ball} as the function $f_0$.
The maximum of the expression above over this class equals
\begin{align*}
\max_{f\in\mathcal{F}}\ |h|^{-1}
\int_{|\mathbf{u}|>1} |h^{-1}\mathbf{u}|^{-p} f(\mathbf{u})\mathrm{d}\mathbf{u} 
 = |h|^{-1}\sup_{|\mathbf{u}|>1} |h^{-1}\mathbf{u}|^{-p}
\end{align*}
Note that the supremum is achieved on the boundary (consider scaling by a scalar $\mathbf{u}:= \lambda \mathbf{u}$). We have then for some $f'\in\mathcal{F}$
\begin{align*}
   K_h\star f'(\mathbf{0}) = \Omega(1) |h|^{-1}\sup_{|\mathbf{u}|=1} |h^{-1}\mathbf{u}|^{-p}
\end{align*}
This is equivalent to 
\begin{align}
   K_h\star f'(\mathbf{0}) = \Omega(1)\cdot |h|^{-1}\left(\inf_{|\mathbf{u}|=1} |h^{-1}\mathbf{u}|\right)^{-p}
\end{align}
We can use the max norm $|\cdot|=|\cdot|_{\infty}$ because of equivalence of all vector norms. Let $\lambda_i$ 
where $|\lambda_1|\geqslant \ldots \geqslant |\lambda_d|$ be the eigenvalues of $h$. Then $\lambda_i^{-1}$ are eigenvectors of $h$. Let $\mathbf{u}$ be the vector such that $h^{-1}\mathbf{u} = \frac{1}{\lambda_1}\mathbf{u}$; it follows that $\inf_{|\mathbf{u}|=1} |h^{-1}\mathbf{u}| \leqslant |\lambda_1|^{-1}$; since $|h| = \prod_{i=1}^d\lambda_i$ one obtains
\begin{align*}
   K_h\star f'(\mathbf{0}) = \Omega(1)\cdot |\lambda_1|^{p}\prod_{i=1}^{d}|\lambda_i|^{-1}
\end{align*}
this finishes the proof.
\end{proof}
From \Cref{eq:lower_bound} it is clear that when $h\to 0$
one needs not only $K$ to decay at least as fast as the negative power of $d$ (with $p<d$ and $\lambda_1 = \ldots = \lambda_d\to 0$
the estimate is unbounded) but also $h$ to keep some balance between the bandwidth eigenvalues. 
We note that in ~\cite{multi_KDE}, for the simpler case of product kernels and diagonal bandwidth, one assumes that $h = (\lambda+o(1))\cdot h_0$ where $h_0$ is a positive diagonal matrix and $\lambda\to 0$; this implies 
that eigenvalues are of comparable order.

\begin{remark}
Note that the kernel in this argument is non-compact, but has moments up to an arbitrary fixed order.
\end{remark}

\subsection{Multivariate KDE bias under general bandwidths}

We give a fairly general bounds on the bias below. 
Note that formulas often cited in the literature, such as 
p. 95 in \cite{wand1994kernel} are limited to compact $K$. The authors suggest that fixing this can be done at the cost of assuming more on the density $f$\footnote{p. 95 in \cite{wand1994kernel} in : "the assumptions of the compact support of $K$ can be removed by imposing more complicated conditions on $f$" }
We show that that extra conditions on $f$ are actually not necessary, and kernels with non-compact support can be handled by the decay.

\begin{theorem}[General bias formula]\label{thm:main}
Let $K$ be a $k$-th order kernel with bounded moments up to $k$.
Suppose that $f$ has $k$-th derivatives bounded in a
$\delta$-neighborhood of $\mathbf{x}'$.
Then the remainder in the bias expansion equals
\begin{align*}
\mathsf{bias}(\hat{f}(\mathbf{x}'))- \sum_{i=0}^{k}\mu_k(\mathbf{x}',h) = R_k(\mathbf{x}',h)\cdot \|h\|^2
\end{align*}
where $\mu_k(\mathbf{x}',h) $ are defined in \Cref{eq:moments}
and for any $\delta$
\begin{align*}
   | R_k(\mathbf{x}',h) |\leqslant 2|h|^{-1}\sup_{|\mathbf{u}|>\delta/\|h\|}K(\mathbf{u}) + C \cdot \mu_K(k)\cdot \frac{1}{k!}\sup_{|\mathbf{u}|\leqslant \delta}\|D^k f(\mathbf{x}'+\mathbf{u}) \|
\end{align*}
where the constant $C$ depends only on the chosen norm.
\end{theorem}
\begin{corollary}[Bias under $k$-th order kernels]
If $K$, $f$ are as in \Cref{thm:main}, $h\to 0$ and
\begin{enumerate}[(a)]
\item $K$ decays at infinity faster than the negative power of $d$
\item $\|h\|^d/|h| = O(1)$   
\end{enumerate}
then the remainder is $o(\|h\|^k)$.
\end{corollary}

\begin{remark}[Balance of eigenvalues]
Note that $\|h\|^d/|h|$ can be easily unbounded (consider diagonal matrix with different entries). In the opposite direction by Hadamard's Inequality~\cite{Hadamard} we have that $|h|/\|h\|^d$ is bounded.
If $\sigma_i$ are eigenvalues of $h$ then $\max_{i=1,\ldots,d}|\sigma_i|\leqslant \|h\|$ and $|h| = \prod_{i=1}^{d}\sigma_i$; thus $|h|/\|h\|^d = O(1)$ implies that all eigenvalues are of same magnitude. 
\end{remark}

\begin{remark}
One can allow for larger discrepancy between $\|h\|$ and $|h|$
with faster decay of the kernel.
\end{remark}

In the proof we will use the multivariate Taylor formula with the integral remainder form. 
To get terms up to the $k$-th order we assume that $k$-th derivatives exist and are locally bounded. It might be possible to further weakened the assumptions, e.g. to ue the Taylor formula when $(k-1)$-th derivatives are absolutely continuous~\cite{ANASTASSIOU2001246}.

\begin{lemma}[Multivariate Taylor's Formula~\cite{taylor,ANASTASSIOU2001246}]\label{lem:taylor}
Let $V$ be a compact convex set in $\mathbb{R}^d$ and let $g:V\rightarrow \mathbb{R}$ have
absolutely continuous $(k-1)$-th derivatives. Then for any $\mathbf{x}\in V$ and $h$ such that $\mathbf{x}+h \in V$
\begin{align*}
    g(\mathbf{x}+h) = g(\mathbf{x}) + \sum_{j=1}^{k}\frac{D^j g(\mathbf{x})(h^{(j)})}{j!} + R_{\mathbf{u}}(h)
\end{align*}
where
\begin{align*}
  R_{\mathbf{x}}(h) = \int_{0}^{1}\frac{(1-t)^{k-1}}{(k-1)!}\left(D^{k}g(\mathbf{x}+t\cdot h) - D^{k}g(\mathbf{x})\right)(h^{(k)})\,\mathrm{d}t.
\end{align*}
\end{lemma}

\begin{proof}[of Theorem]
We split the convolution integral
\begin{align}
    K_h\star f(\mathbf{x'})- f(\mathbf{x'}) = \int K(\mathbf{u}) \left(f(\mathbf{x}'-h\mathbf{u})-f(\mathbf{x}')\right)\,\mbox{d}\mathbf{u}
\end{align}
integral in two regions: $|\mathbf{u}| > \delta$ and $|\mathbf{u}| \leqslant \delta $ where $\delta=\delta(h)$ will depend on $h$. The general strategy is as follows: "big" values of $\mathbf{u}$ are handled by the decay of $K$, whereas "small" are worked out by the smoothness of $f$.
Consider first "big"
\begin{align*}
    I_1 = |h|^{-1}\int_{|\mathbf{u}|>\delta} K(h^{-1}\mathbf{u})\left( f(\mathbf{x}'-\mathbf{u})-f(\mathbf{x}')\right)\,\mathrm{d}\mathbf{u}
\end{align*}
Let $\psi(r) = \sup_{|\mathbf{u}|>r}K(\mathbf{u})$.
By the properties of the matrix norm $|\mathbf{u}|\leqslant \|h\|\cdot |h^{-1}\mathbf{u}|$. Therefore in the region of integration 
$|h^{-1}\mathbf{u}|\geqslant |\mathbf{u}|/\|h\| \geqslant \delta\|h\|^{-1}$ and $K(h^{-1}\mathbf{u})\leqslant \psi(\delta\|h\|^{-1})$ since $\psi$ is decreasing. Since $\int |f| = 1$, we obtain
\begin{align}\label{eq:bound1}
   I_1 \leqslant 2|h|^{-1}\psi\left(\delta \|h\|^{-1}|\right)
\end{align}
Consider now the case of "small" values of $\mathbf{u}$.
We assume that $\delta = o(1)$ so that we can apply the Taylor formula.
The main terms $\mu_k(\mathbf{x}',h) $ are as in \Cref{eq:moments} and are well defined provided 
that $|\mathbf{u}|^j K(\mathbf{u})$ is absolutely integrable and that $D^j f$ exists at $\mathbf{x}'$. It suffices to consider the remainder. Let
\begin{align}
   I_2 = |h|^{-1}\int_{|\mathbf{u}|\leqslant \delta} K(h^{-1}\mathbf{u}) \left|D^{k}f(\mathbf{x}'-t\mathbf{u})(\mathbf{u})^{(k)}\right|\,\mathrm{d}\mathbf{u}
\end{align}
for any fixed $t\in [0,1]$. If we bound
this integral uniformly in $t$, let's say $|I|\leqslant M$ then according to \Cref{lem:taylor} we will get $\frac{1}{k!} 2M$. Let's change variables $\mathbf{v} = t \mathbf{u}$. We have
\begin{align*}
   I_2 = t^{-k-d} |h|^{-1}\int_{|\mathbf{v}|\leqslant \delta t} K(t^{-1}h^{-1}\mathbf{v}) \left|D^{k}f(\mathbf{x}'-\mathbf{v})(\mathbf{v})^{(k)}\right|\,\mathrm{d}\mathbf{v}
\end{align*}
By the properties of multilinear maps 
$$\left|D^{k}f(\mathbf{x}'-\mathbf{v})(\mathbf{v})^{(k)}\right| = O(1)\left\|D^{k}f(\mathbf{x}'-\mathbf{v})\right\|\mathbf{v}|^k$$
where the constant $O(1)$ depends only on the chosen norms. We obtain
\begin{align*}
     I_2 \leqslant O(1)\cdot t^{-k-d} |h|^{-1}\int_{|\mathbf{v}|\leqslant \delta t}  |\mathbf{v}|^{k} K(t^{-1}h^{-1}\mathbf{v})
     \left\|D^{k}f(\mathbf{x}'-\mathbf{v})\right\|\,\mathrm{d}\mathbf{v}
\end{align*}
Now if $\left\|D^{k}f(\mathbf{x}'-\mathbf{v})\right\| \leqslant B(\delta)$ for $|\mathbf{v}|\leqslant \delta$, we obtain
\begin{align*}
     I_2 &\leqslant O(1)\cdot B(\delta)\cdot t^{-k-d} |h|^{-1}\int_{|\mathbf{v}|\leqslant \delta t}  |\mathbf{v}|^{k} K(t^{-1}h^{-1}\mathbf{v})
     \,\mathrm{d}\mathbf{v}\\
     & \leqslant O(1)\cdot B(\delta) \cdot \int_{|h\mathbf{u}|\leqslant \delta }  |h\mathbf{u}|^{k} K(\mathbf{u})
     \,\mathrm{d}\mathbf{u} \\
     & \leqslant O(1)\cdot B(\delta) \cdot \|h\|^k\int_{|h\mathbf{u}|\leqslant \delta } |\mathbf{u}|^k K(\mathbf{u})\, \mathrm{d}\mathbf{u}
\end{align*}
where we changed variables $\mathbf{v}=t\cdot h\mathbf{u}$
and used the norm inequality $|h\mathbf{u}|\leqslant \|h\||\mathbf{u}|$. Since $\|\mathbf{u}|^k K(\mathbf{u})$ is integrable, we obtain
\begin{align}\label{eq:bound2}
    I_2 \leqslant O(1)\cdot B(\delta)\cdot \mu_K(k)\cdot \|h\|^k
\end{align}
The result follows by combining \Cref{eq:bound1} and \Cref{eq:bound2}.

\end{proof}

\bibliographystyle{amsalpha}
\bibliography{citations}

\end{document}